\documentclass[11pt,letterpaper,dvips]{article}

 \usepackage{amssymb,latexsym,amsmath,amsthm}

\newtheorem*{thm*}{Theorem}

\newtheorem{thm}{Theorem}

\newtheorem{remark}{Remark}

\newtheorem{cor}{Corollary}

\newtheorem*{cor*}{Corollary}

\newtheorem*{ack}{Acknowledgement}

\newtheorem*{prop*}{Proposition}

\newtheorem{prop}{Proposition}

\newenvironment{theorem}[2][Theorem]{\begin{trivlist}
\item[\hskip \labelsep {\bfseries #1}\hskip \labelsep {\bfseries #2}]}{\end{trivlist}}

\begin{document}

\def\L{ {\mathcal{L}}}
\def\d{ \partial }
\def\Na{{\mathbb{N}}}

\def\Z{{\mathbb{Z}}}

\def\IR{{\mathbb{R}}}

\newcommand{\E}[0]{ \varepsilon}

\newcommand{\la}[0]{ \lambda}

\newcommand{\s}[0]{ \mathcal{S}}

\newcommand{\AO}[1]{\| #1 \| }

\newcommand{\BO}[2]{ \left( #1 , #2 \right) }

\newcommand{\CO}[2]{ \left\langle #1 , #2 \right\rangle}

\newcommand{\R}[0]{ \IR\cup \{\infty \} }

\newcommand{\co}[1]{ #1^{\prime}}

\newcommand{\p}[0]{ p^{\prime}}

\newcommand{\m}[1]{   \mathcal{ #1 }}

\newcommand{ \A}[1]{ \left\| #1 \right\|_H }

\newcommand{\B}[2]{ \left( #1 , #2 \right)_H }

\newcommand{\C}[2]{ \left\langle #1 , #2 \right\rangle_{  H^* , H } }

 \newcommand{\HON}[1]{ \| #1 \|_{ H^1} }

\newcommand{ \Om }{ \Omega}

\newcommand{ \pOm}{\partial \Omega}

\newcommand{\D}{ \mathcal{D} \left( \Omega \right)}

\newcommand{\DP}{ \mathcal{D}^{\prime} \left( \Omega \right)  }

\newcommand{\DPP}[2]{   \left\langle #1 , #2 \right\rangle_{  \mathcal{D}^{\prime}, \mathcal{D} }}

\newcommand{\PHH}[2]{    \left\langle #1 , #2 \right\rangle_{    \left(H^1 \right)^*  ,  H^1   }    }

\newcommand{\PHO}[2]{  \left\langle #1 , #2 \right\rangle_{  H^{-1}  , H_0^1  }}

 \newcommand{\HO}{ H^1 \left( \Omega \right)}

\newcommand{\HOO}{ H_0^1 \left( \Omega \right) }

\newcommand{\CC}{C_c^\infty\left(\Omega \right) }

\newcommand{\N}[1]{ \left\| #1\right\|_{ H_0^1  }  }

\newcommand{\IN}[2]{ \left(#1,#2\right)_{  H_0^1} }

\newcommand{\INI}[2]{ \left( #1 ,#2 \right)_ { H^1}}

\newcommand{\HH}{   H^1 \left( \Omega \right)^* }

\newcommand{\HL}{ H^{-1} \left( \Omega \right) }

\newcommand{\HS}[1]{ \| #1 \|_{H^*}}

\newcommand{\HSI}[2]{ \left( #1 , #2 \right)_{ H^*}}

\newcommand{\Ov}{ \overline{ \Omega}}
\newcommand{\WO}{ W_0^{1,p}}
\newcommand{\w}[1]{ \| #1 \|_{W_0^{1,p}}}

\newcommand{\ww}{(W_0^{1,p})^*}

\title{Supercritical elliptic problems on a perturbation of the ball}
\author{Craig Cowan  \\
{\it\small Department of Mathematical Sciences}\\
{\it\small University of Alabama in Huntsville}\\
{\it\small 258A Shelby Center}\\
\it\small Huntsville, AL 35899 \\
{\it\small ctc0013@uah.edu} }

\maketitle


\vspace{3mm}

\begin{abstract}   We examine the H\'enon equation $ -\Delta u =|x|^\alpha u^p$ in $ \Omega \subset \IR^N$ with $u=0$ on $ \pOm$   where $ 0 < \alpha$. We show there  exists a sequence $ \{p_k\}_k \subset [ \frac{N+2}{N-2}, p_{\alpha}(N)]$ with $p_1 < p_2 <p_3 < ...$,
   $ p_k \nearrow p_{\alpha}(N)$ such that for any $ \frac{N+2}{N-2} \le  p < p_{\alpha}(N)$, which avoids $ \{p_k\}_k $,  there exists a positive classical solution of the H\'enon equation, provided $ \Omega$ is a sufficiently small  perturbation of the unit ball.

We also examine the Lane-Emden-Fowler equation in the case of an exterior domain; ie.  $ -\Delta u = u^p$ in $ \Omega$, an exterior domain, with $ u=0 $ on $ \pOm$.  We show the existence of  $ \frac{N+2}{N-2} \le p_1 < p_2 < p_3<...$ with $ p_k \rightarrow \infty$ such that if $ \frac{N+2}{N-2} < p$, which avoids $\{p_k\}_k$,
 then there exists
 a positive \emph{fast decay} classical solution,  provided $ \Omega$ is a sufficiently small perturbation of the exterior of the unit ball.

\end{abstract}

\noindent
{\it \footnotesize 2010 Mathematics Subject Classification}. {\scriptsize }

\section{Introduction}

In this note we are interested in the existence of positive classical solutions of the H\'enon equation
 \begin{equation} \label{dom_pe}
 \left\{
 \begin{array}{lcl}
\hfill   -\Delta u  &=& |x|^\alpha u^p \quad  \mbox{ in $\Omega$},   \\
\hfill u&=& 0 \qquad \quad \mbox{on $\pOm$,}
\end{array}\right.
  \end{equation} where $ \Omega$ is a  bounded domain in $ \IR^N$ with $ N \ge 3$   and where $ 0 < \alpha$ and $ 1<p$.

  For $ 1 <p< \frac{N+2}{N-2}$ it is known that $ H_0^1(\Omega)$ is compactly embedded in $L^{p+1}(\Omega)$.   This easily shows that $H_0^1(\Omega)$ is compactly embedded in $L^{p+1}(\Omega, |x|^\alpha dx)$ and hence a standard minimization argument shows the existence of a nonnegative nonzero solution of (\ref{dom_pe}).  Applying the maximum principle and elliptic regularity theory shows that $ u$ is a classical positive solution of (\ref{dom_pe}).     For $ p > \frac{N+2}{N-2}$ one loses the compact embedding of $H_0^1(\Omega)$ into $L^{p+1}(\Omega)$ and the above proof is no longer valid.   This is precisely the case we are interested in.

  In what follows $B$  will always denote the unit ball in $ \IR^N$ centered at the origin, ie. $B:=\{ x \in \IR^N: |x|<1 \}$.
  We begin with a non existence result. A classical Pohozaev argument shows  there is no positive  classical solution of (\ref{dom_pe}) provided $ \Omega$ is a smooth bounded star shaped domain in $ \IR^N$ with $ p > \frac{N+2 + 2 \alpha}{N-2}=:p_\alpha(N)$.   This suggests that one may hope to prove the existence of a positive classical solution of (\ref{dom_pe})  in the case where $ 1<p<p_\alpha(N)$, and indeed one has  the following result,

  \begin{thm*} \textbf{A}. (Ni \cite{Ni})  Suppose  $ N \ge 3$, $ 0  < \alpha$, $ \Omega=B$ and    $ 1 <p< p_\alpha(N)$. Then there exists a positive classical radial solution of (\ref{dom_pe}).

  \end{thm*}

\begin{proof} The idea of the proof is to show that $H_{0,rad}^1(B):=\{ u \in H_0^1(B): u \mbox{ is radial} \}$ is compactly embedded in the weighted space $L^{p+1}(B, |x|^\alpha dx)$ for  $ 1 <p< p_\alpha(N)$.    One can then perform a standard minimization argument to obtain a positive solution of (\ref{dom_pe}).   One should note that in the above approach the radial symmetry of the domain is crucial.

\end{proof}

\begin{remark}   An alternate proof of Theorem A, using a change of variables, is available.   This approach is taken from  \cite{MEMS_pull} (and was also independently noticed in \cite{Grossi_2})  where it was used to analyze various numerically observed phenomena related to the extremal solution associated with equations  of the form
 \begin{equation*}
\left\{ \begin{array}{ll}  
-\Delta u =   \lambda (1+\frac{\alpha}{2})^2|x|^\alpha f(u)  & \hbox{in } B.\\
 \quad \quad \,\, u=0 & \hbox{on }\partial B, \end{array} \right.
\end{equation*}   See the appendix for details.

\end{remark}

After the work of Ni \cite{Ni} the H\'enon equation did not receive much attention until  \cite{SMETS}, where they examined (\ref{dom_pe}) in the case of  $ \Omega=B$.  They showed, among many results, that for $ 1 <p< \frac{N+2}{N-2}$ the ground state solution is non radial provided $ \alpha>0$ is sufficiently large.  Since this work there has been many related works,  see \cite{Badiale, Byeon, Cao, smets}, which show various results regarding properties of solutions to (\ref{dom_pe}) in the case where $ \Omega=B$. Some of these works include certain ranges of $ p > \frac{N+2}{N-2}$.

      After the majority of this work was completed we learned of the  recent interesting work \cite{Grossi_1} where they examine (\ref{dom_pe}) for general bounded domains containing the origin.   They show many interesting results, one of which is the existence of positive solutions provided $ p = \frac{N+2+2 \alpha}{N-2}-\E$ where $ \E>0$ is small.   In addition they have another recent preprint \cite{Grossi_2} where they examine (\ref{dom_pe}) on $ \IR^N$ and obtain many interesting results.   We also mention the very interesting related works \cite{add_1,add_2} from which we borrowed many ideas.

   We now state our first result.
  \begin{thm}  \label{main}  Suppose $ 0 < \alpha$.  Then there exists a sequence \\ $ \{p_k\}_k \subset [ \frac{N+2}{N-2}, p_{\alpha}(N)]$ with $p_1 < p_2 <p_3 < ...$,
   $ p_k \nearrow p_{\alpha}(N)$ such that for any $ \frac{N+2}{N-2} \le p < p_{\alpha}(N)$, which avoids $ \{p_k\}_k $,  there exists a positive classical solution of (\ref{dom_pe}), provided $ \Omega$ is a sufficiently small smooth perturbation of the unit ball (which we make more precise later).

  \end{thm}

  Of course it would  desirable to extend the above result to some non symmetric  domains  which contain the origin but which are not perturbations of the unit ball.   In Proposition \ref{non-origin} we show there are domains such that (\ref{dom_pe}) does not have a positive classical solution.  This result is fairly trivial since the domains we consider don't contain the origin.   It would be interesting to see if one can obtain non existence results on domains which contain the origin.     Of course one would could not obtain a result for the full range of $ p < p_{\alpha}(N)$ after considering the above mentioned results obtained in \cite{Grossi_1}.

   We now discuss some equations related to (\ref{dom_pe}) and which have received attention recently.    Suppose that $u(x)$ is a positive solution of (\ref{dom_pe}) on some bounded connected domain $ 0 \in \Omega$ and we  let $ v(x)$ denote the Kelvin transform of $u(x)$, ie. $ v(x)= u(\frac{x}{|x|^2})$. Then $v$ is a solution of

   \begin{equation} \label{exterior}
 \left\{
 \begin{array}{lcl}
\hfill   -\Delta v(x)  &=& |x|^{ \beta} v(x)^p \quad  \mbox{ in $\tilde{\Omega}$},   \\
\hfill u&=& 0 \qquad \quad \quad  \mbox{on $\partial \tilde{\Omega}$,}
\end{array}\right.
  \end{equation}  where $\beta:= (N-2)p-N-2-\alpha$ and $ \tilde{\Omega}$ is the transformed domain, which is an exterior domain.     Also note that since $u$ is a classical solution we have
  $v(x)  \le \frac{C}{|x|^{N-2}}$ for all $ x \in \tilde{\Omega}$.

  Recently there has been a lot of interest in supercritical Lane-Emden-Fowler equations on exterior domains, ie. equation (\ref{exterior}) in the case where $ \beta=0$,  $ \tilde{\Omega}$ is an exterior domain in $ \IR^N$ and $ p > \frac{N+2}{N-2}$; see \cite{fast_slow, Davila_del_first, delpino_1}.     Note that when $ p > \frac{N}{N-2}$ we have $ \frac{2}{p-1} < N-2$ which motivates calling a solution $v$ of (\ref{exterior}) which satisfies $O(|x|^\frac{-2}{p-1})$ (resp. $O(|x|^{2-N})$) at infinity a \emph{slow decay} (resp. \emph{fast decay}) solution.     In \cite{fast_slow} it was shown there exists infinitely many slow decay solutions of (\ref{exterior}) for any $ p > \frac{N+2}{N-2}$ and there exists a fast decay solution for $  p > \frac{N+2}{N-2}$ but sufficiently close to the critical exponent.

We now state out second result.
\begin{thm} \label{main_2} There exists $ \frac{N+2}{N-2} \le p_1 < p_2 <p_3 < ...$ with $ p_k \rightarrow \infty$ such that for all $ \frac{N+2}{N-2} <p $,  which avoids $ \{p_k\}_k$, there exists a smooth positive  \emph{fast decay} solution of (\ref{exterior}) in the case of $ \beta=0$ provided $ \tilde{\Omega}$ is a smooth exterior domain which is a sufficiently small  perturbation of the exterior of the unit ball.
\end{thm}

We now discuss another related work.
In \cite{delpino_small}  they examined
   \begin{equation}
 \left\{
 \begin{array}{lcl} \label{small}
\hfill   -\Delta u(x)  &=&  u(x)^p \quad  \mbox{ in $\Omega:=\mathcal{D} \backslash B_\delta(P)$},   \\
\hfill u&=& 0 \qquad \quad \quad  \mbox{on $\partial \Omega$,}
\end{array}\right.
  \end{equation}  where $ P \in \mathcal{D}$ is a smooth bounded domain in $ \IR^N$ and where $ \delta>0$ is small.   They examine the supercritical problem $ \frac{N+2}{N-2}<p$ and they find a sequence $ p_1 < p_2 < p_3<...$ with $ p_k \nearrow \infty$ such that if $p$ is given with $ p \neq p_k$ for all $k$, then for all $ \delta>0$ sufficiently small, (\ref{small}) has a positive classical solution.

A portion of our approach follows \cite{delpino_small} very closely, which is not surprising since, as noted by many authors,  the equation $ -\Delta u = |x|^\alpha f(u) $ in $ 0 \in \Omega$,   has some similarities to $ -\Delta u = f(u)$ on an annular domain with center $0$.   In fact they end up examining the radial solution of  (\ref{dom_pe}) on the unit ball in the case of a specific value of  $ \alpha = \alpha_p$.  They then go on to show this solution is non degenerate by proving an analyticity result.   Their prove can be extended to show the same analyticity result in our case and our method then diverges from theirs.

\begin{remark} As a general comment we point out that recently problems of the form $L(u)=  |x|^\alpha f(u)$ in $ \Omega $ have attracted a lot of interest.  In \cite{GG,EGG} it was shown that the term $ |x|^\alpha$ can drastically alter the so called \emph{critical dimension} associated with the problem $ -\Delta u = \lambda f(u)$ in $B$.  To state these results another way,  they showed that the term $ |x|^\alpha$ can have a major impact in  the regularity of stable and finite Morse index solutions.  They also showed Liouville type theorems for finite Morse index theorems of similar equations on $ \IR^N$.       One can also examine the possibility of obtaining Liouville theorems for $-\Delta u = |x|^\alpha u^p$ in $ \IR^N$,  but now without any additional spectral assumptions.     One expects to obtain the result that for any $ p < p_\alpha(N)$ there is no positive solution.  This was known to be true in the radial case but has only recently been proven in the case of $N=3$, see \cite{phan}.  They showed partial results in higher dimensions.  The major difficultly here is that one does not have the moving plane method (we are only considering the case of positive $\alpha$).   After this work their methods were used to examine the case of fourth order problems and systems.  See \cite{craig_11, Fazly_ghouss, Phan_syst}.

\end{remark}

The paper is organized in the following way.   In Section 2 we  prove Theorem 1',  which is Theorem 1 up to  a non degeneracy condition.  In Section 3 we obtain the non degeneracy condition.       Theorem 2 is obtained from Theorem 1 after an application of the Kelvin transform.

\begin{ack} In the original version of this paper we proved there exists a countable set $ \{p_k\}_k \subset [ \frac{N+2}{N-2}, p_\alpha(N)] $ such that for any $ p \in [ \frac{N+2}{N-2}, p_\alpha(N)]$, which avoided $ \{p_k\}_k$, there exists a positive classical solution of (\ref{dom_pe}), provided $ \Omega$ is a sufficiently small perturbation of the unit ball.     The anonymous referee suggested this result could be improved to the statement in Theorem \ref{main} and we greatly appreciate their suggestion.

\end{ack}

\section{The fixed point argument}

  For $ 1 <p<p_\alpha(N)$ we let $ v_p(r)$ denote the radial positive classical solution of (\ref{dom_pe}) on the unit ball promised by  Theorem A.     In this section we obtain a  positive classical solution of (\ref{dom_pe}), in the case where $ \Omega$ is a small perturbation of the unit ball, provided the radial solution $ v_p$  is non degenerate.

   We begin by using a change of variables from \cite{gelfand_sing},  their interest was in singular stable solutions of $-\Delta u = \lambda e^u$ in $ \Omega$ with $ u=0$ on $ \pOm$ where $ \Omega$ was a perturbation of the unit ball.
  Let $ \psi: \overline{B} \rightarrow \IR^N$ be a smooth map and for $ t >0$ define
  \[ \Omega_t:= \left\{ x + t \psi(x): x \in B \right\},\]  which, for small $t$, will be the small perturbation of the unit ball that we solve (\ref{dom_pe}) on.    There is some small $ 0 <t_0$ such that for all $ 0 <t<t_0$ one has that $ \Omega_t$ is diffeomorphic to the unit ball $B$.  Let $ y=x +t \psi(x)$ for $ x \in B$ and note there is some $ \tilde{\psi}$ smooth such that $ x = y + t \tilde{\psi}(y) $  for $ y \in \Omega_t$.  Given $u(y)$ defined on $ y \in \Omega_t$ or $v(x)$ defined on $x \in B$ we define the other via $ u(y) =v(x)$.   A computation shows, see \cite{gelfand_sing}, that to find  positive classical solution $u(y)$ of (\ref{dom_pe}) on $ \Omega_t$ it is equivalent  to find a positive classical solution $v(x)$ of

   \begin{equation} \label{ball}
 \left\{ \begin{array}{lcl}
\hfill   -\Delta v - L_t(v)   &=& |x + t \psi(x)|^\alpha v^p \qquad B,   \\
\hfill v&=& 0 \qquad\qquad \qquad \quad  \partial B,
\end{array}\right.
  \end{equation}
  where
  \[ L_t(v) = 2t \sum_{i,k} v_{x_i x_k} \partial_{y_i} \tilde{ \psi}_k + t \sum_{i,k} v_{x_k} \partial_{y_i y_i} \tilde{\psi}_k + t^2 \sum_{i,j,k} v_{x_j x_k} \partial_{y_i} \tilde{\psi}_j  \partial_{y_i} \tilde{\psi}_k. \]
   Note that $v_p$ solves (\ref{ball}) when $ t=0$.

     There are two options to obtain the existence of positive classical solutions to (\ref{ball}) for small $ 0<t$.  We can either  apply the Implicit Function Theorem (IFT) or we can apply a fixed point argument.   To apply the IFT we define
   $ F:\IR \times C_0^{2,\delta}(\overline{B}) \rightarrow C^{0,\delta}( \overline{B})$ by
\[ F(t,v):= -\Delta v - L_t(v) - |x+ t \psi(x)|^\alpha |v(x)|^p,\]
  where $ C_0^{2,\delta}(\overline{B}):= \{ v \in C^{2,\delta}(\overline{B}): v|_{\partial B}=0 \}$.   A computation shows that for $ \alpha \ge 1$ and $ 0 < \delta$ sufficiently small,  that $F$ is a $C^1$ map.
We now note that $F(0,v_p)=0$ and hence,  if one can apply the IFT,  then there is some small $ 0 <t_1$ and  a  continuous  map $[0,t_1) \ni t \mapsto v(t) \in C_0^{2,\delta}( \overline{B})$ such that $ F(t, v(t))=0$ on $[0,t_1)$.   Note that $ v(t)>0$ for small $t$ after considering $v_p$ is positive and $ v(t) \rightarrow v_p$ in $ C^{2,\delta}$.  Of course to apply the IFT we require that
  $ D_v F(0,v_p):C_0^{2,\delta}(\overline{B}) \rightarrow C^{0,\delta}( \overline{B})$ has a bounded inverse.
       To show the existence of the bounded inverse it is sufficient to show the radial solution is non degenerate, ie.  there  is no non zero $ \phi \in C_0^{2,\delta}(\overline{B})$ such that
  \begin{equation} \label{deg}
  -\Delta \phi = p r^\alpha v_p(r)^{p-1} \phi  \qquad \mbox{in $ B$}.
  \end{equation}   The drawback of this approach is that the regularity of $F$ requires  $ \alpha \ge 1$ instead of the more natural assumption that $ \alpha >0$.

   The second approach is to use a fixed point argument.  Now recall  that a positive classical solution of (\ref{ball}) is equivalent to a positive classical solution of (\ref{dom_pe}).  With this in mind we now prove Theorem \ref{main} up to a non degeneracy condition, which is given by the following theorem.  In the next section we show that one has this non degeneracy condition for all but a countable number of $p$ in the desired range.

      \begin{theorem}{1'} (Theorem \ref{main} up to a non degeneracy condition) Let $ \alpha >0$, $ \frac{N+2}{N-2} \le p < p_\alpha(N)$ and suppose there is no non zero solution of (\ref{deg}).  Then for $ 0<t$ sufficiently small there exists a positive classical solution of (\ref{ball}).
      \end{theorem}

   \begin{proof}
     We take the standard approach of linearizing around the approximating solution, $v_p$.   Instead of solving (\ref{ball}) directly we solve  the slightly modified version  given by
   \begin{equation} \label{absolute}
   -\Delta v- L_t(v)= |x+ t \psi(x)|^\alpha |v|^p \qquad B,
   \end{equation} with $v=0$ on $\partial B$.  One can then argue that the solution is positive and hence solves (\ref{ball}).     To solve (\ref{absolute}) we look for solutions of the form $ v = v_p+ \phi$.    One then sees they need to solve
  \begin{equation} \label{nonlinear}
  L(\phi)= L_t(v_p)+ L_t(\phi) + H_t(x, \phi) \qquad B,
   \end{equation} with $ \phi=0 $ on $ \partial B$  where $ L(\phi) = -\Delta \phi - p |x|^\alpha v_p^{p-1} \phi$  and
   \[ H_t(x,\phi)= |x+t \psi(x)|^\alpha |v_p + \phi|^p - |x|^\alpha v_p^p - p |x|^\alpha v_p^{p-1} \phi.\]   Since we are assuming there is no   non zero solution of (\ref{deg}) on can easily show that  $L$ is a one to one and onto continuous linear operator from $ C_0^{2,\delta}(B)$ to $ C^{0,\delta}(B)$, provided $ \delta >0$ is sufficiently small, and hence has a continuous inverse.    So to solve (\ref{absolute}) it is sufficient to find a fixed point of
   \[ T(\phi):= L^{-1}(L_t(v_p))+ L^{-1}( L_t(\phi)) +L^{-1}(H_t(x,\phi)).\] We will show that for sufficiently small $t$ and $R$ that $T$ is a contraction on the closed ball of radius $R$ centered at the origin in $C_0^{2,\delta}(\overline{B})$,  which we denote by $B_R$, and hence has a fixed point.

       Let $0<R_p<1$ be sufficiently small such that for all $ \phi \in B_{R_p}$ one has: $ v_p + \phi \ge \frac{v_p}{2}$ and $ \| v_p^{-1} \phi \|_\delta < \frac{1}{4}$.  Here $ \| \cdot \|_\delta$ is the norm on $ C^{0,\delta}(\overline{B})$ given by $ \| \phi\|_\delta = \| \phi\|_{L^\infty} + [ \phi ]_\delta$,  where the second term is the usual semi norm term.   We  can also assume that $ 0<t_0<1$ where this was previously defined.   Let $ R<R_p$, $ t<t_0$ and $ \phi \in B_R$.  $C$ will denote universal constants which are independent of $R,t,\phi$ but may depend on various parameters including norms of $v_p$. $C(\tau)$ will denote a constant that depends on $\tau$ and which satisfies $ \lim_{\tau \searrow 0} C(\tau)=0$.
     First note that by the continuity of $L^{-1}$ there is some $ C>0$
   \begin{eqnarray*}
    C \| T(\phi)\|_{2,\delta} & \le & \| L_t(v_p)\|_\delta + \|L_t(\phi)\|_\delta + \|H_t(x,\phi)\|_\delta \\
    & \le &  C t \|v_p\|_{2,\delta} + C t \| \phi\|_{2,\delta}  + \| H_t(x,\phi)\|_{\delta},
    \end{eqnarray*}

     where $\| \cdot \|_{2,\delta}$ is the norm on $C^{2,\delta}$.   Now note we can rewrite $H_t$ as
    \[ H_t(x,\phi)=\left\{ |x+t \psi(x)|^\alpha-|x|^\alpha \right\} (v_p+\phi)^p + |x|^\alpha \left\{ (v_p+\phi)^p-v_p^p - p v_p^{p-1} \phi \right\}.  \]  One then obtains
    \[ \| H_t(x,\phi)\|_{L^\infty} \le C(t) + \| (v_p+\phi)^p-v_p^p - pv_p^{p-1} \phi \|_{L^\infty},\]  and the last term can be bounded above by a term of the form $  C(R) \| \phi\|_{L^\infty}$ by using, for instance, the Binomial Theorem.

     We now compute the semi norm of $H_t(x,\phi)$ and again we use the Binomial Theorem.  Recall that $ (a+b)^p= \sum_{k=0}^\infty \gamma_k a^{p-k} b^k$,  where $ \gamma_k$ are the binomial coefficients, and where this series converges absolutely provided $ |b| <a$.   The following computations will utilize the following two inequalities: $ [fg]_\delta \le \| f \|_{L^\infty} [g]_\delta + [f]_\delta \| g \|_{L^\infty}$ and $ [f^k]_\delta \le k \|f\|_\delta^k$ for $k$ a positive integer.  Using the above estimates we see that
    \begin{eqnarray*}
     \left[ H_t(x,\phi) \right]_\delta & \le &  C(t) + C \| (v_p+\phi)^p - v_p^p-pv_p^{p-1} \phi\|_{L^\infty} \\
      && + \left[ (v_p+\phi)^p - v_p^p-pv_p^{p-1} \phi \right]_\delta
      \end{eqnarray*} and now recall the second term is bounded above by $ C(R) \| \phi\|_{L^\infty}$.   We now compute the last term.    Firstly we write
      \[ (v_p+\phi)^p - v_p^p-pv_p^{p-1} \phi = \sum_{k=2}^\infty \gamma_k v_p^{p-k} \phi^k,\] and hence the last term is bounded above by
     $ \sum_{k=2}^\infty |\gamma_k| \left[ v_p^{p-k} \phi^k \right]_\delta.$  We write $ \left[ v_p^{p-k} \phi^k \right]_\delta $ as $ \left[ v_p^{p-1} \phi \left( \frac{\phi}{v_p} \right)^{k-1} \right]_\delta$ and then we expand this using the two rules mentioned above.  Doing this gives
     \begin{eqnarray*}
       \left[ v_p^{p-k} \phi^k \right]_\delta & \le & \| \phi\|_{L^\infty} \| v_p\|_{L^\infty}^{p-1} (k-1) \| v_p^{-1} \phi\|_\delta^{k-1} + [ \phi ]_\delta  \|v_p\|_{L^\infty}^{p-1} \| v_p^{-1} \phi\|^{k-1} \\
       && + \| \phi\|_{L^\infty} [ v_p^{p-1}]_\delta \| v_p^{-1} \phi\|_{L^\infty}^{k-1}.
       \end{eqnarray*}  Using this we see that we have
       \[ \left[ (v_p+\phi)^p - v_p^p-pv_p^{p-1} \phi \right]_\delta \le C(R) \| \phi\|_\delta,\]  and combining this with the $L^\infty$ estimate gives $ \| H_t(x,\phi)\|_\delta \le C(t) + C(R) \| \phi\|_\delta$.  Using this we see that for all $ t<t_0, R<R_p$ and $ \phi \in B_R$ we have
       \begin{equation} \label{into}
        \|T(\phi)\|_{2,\delta} \le C t  + C(t) + \left\{C t + C(R) \right\} \| \phi\|_{2,\delta}.
       \end{equation}   We will pick the parameters $ R$ and $t$ after we find the sufficient condition for $T$ to be a contraction.

       Let $ t<t_0, R<R_p$ and $ \phi_0,\phi \in B_R$.   Using the continuity of $L^{-1}$ and the form of $L_t$ gives the estimate
       \[ \| T(\phi_0)- T(\phi)\|_{2,\delta} \le C t \| \phi_0-\phi\|_{2,\delta} + C \|H_t(x,\phi_0)-H_t(x,\phi)\|_\delta.\]

 We now rewrite the term $ H_t(x, {\phi_0})-H_t(x,\phi)$ as

    \begin{eqnarray*}
     I_1+I_2 &=&|x+t \psi(x)|^\alpha \left\{ (v_p+ {\phi_0})^p -(v_p+\phi)^p - p  v_p^{p-1}({\phi_0}-\phi) \right\} \\
      && + \left\{ |x+t \psi(x)|^\alpha - |x|^\alpha \right\} p  v_p^{p-1} ({\phi_0}-\phi).
     \end{eqnarray*}  One easily sees that $\| I_2 \|_\delta \le C(t) \| \phi_0- \phi\|_\delta$.
       Using the Binomial Series approach and the fact that there is some  $ C>0$ such that $ \| v^{-1} (\phi_0-\phi)\|_{L^\infty} \le C \| \phi_0- \phi\|_{0,1}$ one obtains
      $\|I_1 \|_{L^\infty} \le C(R) \| \phi_0 - \phi\|_{0,1} \le C(R) \| \phi_0 - \phi\|_{2,\delta}$.
       A calculation shows that
     \[ \left[ I_1 \right]_\delta \le C(R) \| \phi_0 - \phi\|_{2,\delta} + C \left[ (v_p+\phi_0)^p - (v_p+\phi)^p - pv_p^{p-1} (\phi_0-\phi) \right]_\delta,\]  and we now estimate the final semi norm  term.  Towards this define $J=
      (v_p+\phi_0)^p - (v_p+\phi)^p - pv_p^{p-1} (\phi_0-\phi)$ and we can rewrite this in terms of the Binomial Series as
      \[ J = (\phi_0-\phi) \sum_{k=2}^\infty \gamma_k \sum_{i=0}^{k-1} v_p^{p-1} \left( \frac{ \phi_0}{v_p} \right)^{k-1-i} \left( \frac{ \phi}{v_p} \right)^i, \]  and from this representation one can show $ [J]_\delta \le C(R) \| \phi_0 - \phi\|_{2,\delta}$.   Combining the results we see that $ \| T(\phi_0)-T(\phi)\|_{2,\delta} \le ( C(t)+C(R)) \| \phi_0-\phi\|_{2,\delta}$.  Using this and (\ref{into}) one sees that $T$ will be a contraction mapping on $B_R$ provided $ R<R_p$ is fixed sufficiently small and then $ t<t_0$ is chosen sufficiently small.
     \end{proof}

     We now give a non existence result for positive solutions of (\ref{dom_pe}).
  \begin{prop} \label{non-origin}  Let $ 0 \in \Omega$ denote a smooth bounded domain in $ \IR^N$ which is star shaped with respect to the origin,  $ 0 < \alpha$ and  suppose  $ |x_m| \rightarrow \infty$.  Then for any  $ \frac{N+2}{N-2} <p$ there is no positive classical solution of (\ref{dom_pe}) on $ \Omega_m:=x_m + \Omega$ provided $m$ is sufficiently big.

  \end{prop}

  \begin{proof}  Suppose $ u$ is a positive classical solution of (\ref{dom_pe}) on $ \Omega_m$.  Define  $ t_m, z_m, \gamma_m$ by $ t_m^{p-1} |x_m|^\alpha=1$, $ z_m = \frac{x_m}{|x_m|}$ and $ \gamma_m = |x_m|^{-1}$.   Note that $ t_m, \gamma_m \rightarrow 0$ and $ |z_m|=1$.   Define the rescaled functions $ v_m(x)= t_m^{-1} u(x_m+x)$ and note that $v_m$ is a positive solution of
   \begin{equation*}  \label{scaled}
 \left\{
 \begin{array}{lcl}
\hfill   -\Delta v  &=& |z_m+\gamma_m x|^\alpha v^p \quad  \mbox{ in $\Omega$},   \\
\hfill v&=& 0 \qquad \quad \qquad  \qquad \mbox{on $\pOm$.}
\end{array}\right.
  \end{equation*}   Applying a Pohozaev argument gives

  \begin{eqnarray*}
  \int_\Omega v_m^{p+1} |z_m+\gamma_m x|^\alpha \left( \frac{N-2}{2} - \frac{N}{p+1} - \frac{\E_m(x)}{p+1} \right)  && \\
  + \frac{1}{2} \int_{\pOm} | \nabla v_m|^2 x \cdot \nu(x) && =0
  \end{eqnarray*}  where $ \E_m(x) =\frac{ x  \cdot \nabla_x (|z_m +\gamma_m x|^\alpha) }{ |z_m+\gamma_m x|^\alpha}$ and $ \nu(x)$ is the outward pointing normal on $ \pOm$.  Now note that $ \E_m(x) \rightarrow 0$ uniformly on $ \Omega$ and hence the result follows after considering the fact $ \frac{N-2}{2} - \frac{N}{p+1}>0$ and $ x \cdot \nu(x) \ge 0$ on $ \pOm$.
  \end{proof}

\section{Non degeneracy of the radial solution}

 We begin by defining a few necessary quantities.   Let $ \Delta_\theta$ denote the Laplace-Beltrami operator on $ S^{N-1}$ and we let $ (\phi_k,\lambda_k)$ denote the associated eigenpairs, ie.  $ -\Delta_\theta \phi_k(\theta) = \lambda_k \phi_k(\theta)$ in $ S^{N-1}$.  Note that $ \lambda_0=0$, $ \phi_0=1$ and $ \lambda_1=\lambda_2=,...,=\lambda_N=N-1, < \lambda_{N+1},..$.   Define the operator $ \tilde{L}$ on the  radial functions by
\[ \tilde{L}(\psi):= r^2 \left( -\Delta \psi - p r^\alpha v_p(r)^{p-1} \psi \right),\]  and we work in the weighted $L^2$ space given by $ H:=L^2(B, |x|^{-2} dx)$.  We define the domain $ \tilde{L}$ by
\[ D(\tilde{L}):=\left\{ \psi \in H \cap H^1_{0,rad,loc}(\overline{B} \backslash \{0\}): \tilde{L}(\psi) \in H \right\}.\]  Note that $ \tilde{L}$ is self-adjoint on $H$.   For given $ \alpha >0$ and $ p>1$ we define
\[ \nu(p):= \inf_{ \psi \in D(\tilde{L})}  \frac{  \int_B |  \nabla \psi|^2 - p r^\alpha v_p(r)^{p-1} \psi^2 dx}{  \int_B \frac{\psi^2}{|x|^2} dx }.\]  Note that $ \nu(p)$ is just the first eigenvalue of $ \tilde{L}$ which is clearly negative after one considers taking $ \psi=v_p$.

  \begin{prop} \label{no_non_zero}  Suppose that (\ref{deg}) has a non zero solution $ \phi$.  Then $ \nu(p) \in \{ - \lambda_k:  k \ge 1\}$.
  \end{prop}

  \begin{proof}

Suppose $ \phi$ satisfies (\ref{deg}) and we write
 $ \phi(x) = \sum_{k=0}^\infty a_k(r) \phi_k(\theta)$.
  We now use the formula $ \Delta \phi = \phi_{rr} + \frac{N-1}{r} \phi_r + \frac{1}{r^2} \Delta_\theta \phi$ to see that
  \[ 0= \sum_{k=0}^\infty \left( a_k''(r) + \frac{N-1}{r} a_k'(r) - \frac{\lambda_k }{r^2} a_k(r) + p r^\alpha v_p(r)^{p-1} a_k(r) \right) \phi_k(\theta),\]  and hence we have
  \begin{equation} \label{eigen}
 0=a_k''(r) + \frac{N-1}{r} a_k'(r) +  p r^\alpha v_p(r)^{p-1} a_k(r) - \frac{\lambda_k }{r^2} a_k(r)  \qquad \forall k \ge 0,
\end{equation}  along with the boundary condition $ a_k(1)=0$. \\
 We begin with $k=0$.    Since $ \lambda_0=0$ we have $ a_0$ is a radial solution of
 \[-\Delta a_0 = p r^\alpha v_p^{p-1} a_0 \qquad \mbox{ in $B$,} \]  with $ a_0(1)=0$.   We now define $\tilde{v_p}$ and $ \tilde{a_0}$ by
\[ v_p(r)= (1+ \frac{\alpha}{2})^\frac{2}{p-1} \tilde{v_p}(r^{\frac{\alpha}{2}+1}), \quad a_0(r)= \tilde{a_0}(r^{\frac{\alpha}{2}+1}),\] as in Theorem B and Corollary \ref{linearized_change_cor} in the appendix.   From Theorem B and Corollary \ref{linearized_change_cor}  we have
\[ -\Delta_{N(\alpha)} \tilde{v_p} = \tilde{v_p}^p \qquad \tilde{B}, \] and
\[ -\Delta_{N(\alpha)} \tilde{a_0} = p \tilde{v_p}^{p-1} \tilde{a_0} \qquad \tilde{B},\]  with zero Dirichlet boundary conditions.  Here $\Delta_{N(\alpha)} $ is the Laplacian in the fractional dimension $ N(\alpha)$ see the appendix.  We now wish to show that $ \tilde{a_0}=0$.   We now recall the following result, from \cite{pacella}, where  it is shown that any positive classical (and hence radial) solution of
   \begin{equation*}
 \left\{
 \begin{array}{lcl} \label{ext}
\hfill   -\Delta w &=& w^p  \quad  \mbox{ in $B$},   \\
\hfill w&=& 0 \qquad   \mbox{on $\partial B$,}
\end{array}\right.
  \end{equation*} is non degenerate, ie. there is no non zero $ \psi $ which satisfies $ -\Delta \psi = p w^{p-1} \psi $ in $B$, with $ \psi=0 $ on $ \partial B$.   A careful examination of the proof shows that it is still valid in fractional dimensions provided $ \psi$ is radial and hence $ \tilde{a_0}=0$ and so $ a_0=0$.   Of course this does not say that $ v_p$  is a non degenerate solution of (\ref{dom_pe}),  but it does say that $ v_p$ is a non degenerate solution of (\ref{dom_pe}) in $H^1_{0,rad}(B)$.

  We now suppose that some $ a_k(r)$ is non zero solution of (\ref{eigen}) for some $ k \ge 1$, which we rearrange to read
  \[ r^2 \left( -\Delta a_k - p r^\alpha v_p^{p-1} a_k \right)= - \lambda_k a_k,\]  and so using the above notation we see that $ \tilde{L}(a_k) = - \lambda_k a_k$, so $ -\lambda_k <0$ is an eigenvalue of $ \tilde{L}$.   We now argue that $ \tilde{L}$ has only one negative eigenvalue and hence $ -\lambda_k$ must be the first eigenvalue of $ \tilde{L}$.

    To see this lets  suppose that $ \tilde{L}$ has at least 2 negative eigenvalues say $ \mu_1 < \mu_2 <0$ since the first is simple and we assume that $b_k$ are the associated eigenfunctions  of $ \tilde{L}$.   Since $ b_1$ is orthogonal to $ b_2$ in $ L^2_{rad}(B, |x|^{-2} dx)$ we have that $ \int_B \nabla b_1 \cdot \nabla b_2 - p \int_B r^\alpha v_p(r)^{p-1} b_1 b_2=0$.  From this we see that \[ I(\psi):= \int_B | \nabla \psi|^2 - p \int_B r^\alpha v_p(r)^{p-1} \psi^2,\] is negative on $ X:= \{ s b_1 + t b_2: s,t \in \IR\}$ except at the origin and hence the radial Morse index of  $ v_p$ is at least two.   Since $ v_p$ is a mountain pass solution of a $C^2$ functional over $ H_{0,rad}^1(B)$ we know that its  radial Morse index is at most one and hence we have a contradiction.  From this we see that  $ \tilde{L}$ has at most one negative eigenvalue.      Recalling that $ \nu(p)<0$ we see that $ \tilde{L}$ has exactly one negative eigenvalue,  and hence we must have  $ \nu(p) \in \{ -\lambda_k: k \ge 1 \}$.

  \end{proof}

  Our goal is to now show $ p \mapsto \nu(p)$ is a non constant analytic function.
  \begin{prop} \label{anal} Let $ 0 < \alpha$ and  define $ \nu(p)$ as above on $ (1, p_\alpha(N))$.  Then $ \nu(p)$ is non constant on this interval.

  \end{prop}

   \begin{proof}
   First recall that for all $ 1<p < p_\alpha(N)$ we have $ \nu(p)<0$.  We now show that $ \lim_{p \searrow 1} \nu(p)=0$ and hence $\nu(p)$ cannot be a constant function.    Observe that $ \frac{-\Delta v_p}{v_p} = r^\alpha v_p^{p-1}$ and since $ v_p$ is positive we have
  \[ \int_B | \nabla \psi|^2 - \int_B r^\alpha v_p^{p-1} \psi^2 \ge 0,\] for all $ \psi \in C_c^\infty(B)$.   From this we conclude that
  \[ \int_B | \nabla \psi|^2 - \int_B p r^\alpha v_p^{p-1} \psi^2 \ge (1-p) \int_B r^\alpha v_p^{p-1} \psi^2,\] for all $ \psi \in C_c^\infty(B)$.   Suppose now we can show that $ v_p^{p-1}$ is uniformly bounded in $B$ for $  p \in (1,1+\E)$. Then   it is easy to see, using the above inequality,  that $ \lim_{p \searrow 1} \nu(p)=0$ and we would be done.   Let $ p_k \searrow 1$  and set
 $ T_k:=v_{p_k}(0)= \| v_{p_k}\|_{L^\infty}$ and define the rescaled functions
  \[  w_k(x):=  \lambda_k^\frac{2+\alpha}{p_k-1} v_{p_k}( \lambda_k x), \quad  |x|<R_k:= T_k^\frac{p_k-1}{2+\alpha},\] where  $ \lambda_k^\frac{2+\alpha}{p-1} T_k =1$.     Now note that $ -\Delta w_k = |x|^\alpha w_k^{p_k}$ in $ B_{R_k}$ and $ 0 < w_k \le w_k(0)=1$ in $ B_{R_k}$.   Our goal is to now show that $ T_k^{p_k-1}$ is bounded.  So towards a contradiction suppose, after passing to a subsequence, that $ R_k \rightarrow \infty$.   By using  the usual harmonic decomposition and a diagonal argument, coupled with the fact that $ p_k \searrow 1$,  we see there is some $ w >0$ in $ \IR^N$ which satisfies $ -\Delta w = |x|^\alpha w$. From this we can conclude that for all smooth compactly supported $ \psi$ we have
    \[ \int |x|^\alpha \psi^2 \le \int | \nabla \psi|^2,\]  and from this we can easily get a contradiction by taking $ 0 \le \psi_R \le 1$ which are supported in $B_{2R}$ with $ \psi_R=1$ on $B_R$ and letting $ R \rightarrow \infty$.   From this we can conclude that $ T_k$ is bounded.
  \end{proof}

We would now like to show $ p \mapsto \nu(p)$ is analytic.  This result is essentially contained in \cite{delpino_small}.  Recall they are examining solutions of (\ref{exterior}) in the case of $ \beta=0$.  Let $ \frac{N+2}{N-2}<p$.  A computation shows that $ p < \frac{N+2+2( p(N-2)-N-2)}{N-2}$ and hence there exists a smooth positive radial solution $v_p$  of
 $ -\Delta v = |x|^{p(N-2)-N-2} v^p$ in $B$ with $v=0$ on $ \partial B$.    Using the methods developed in \cite{Dancer_1} they show that if $ \nu_1(p)$ is defined exactly as we defined $ \nu(p)$, except with $ \alpha= p(N-2)-N-2$,  then $ p \mapsto \nu_1(p)$ is analytic and non constant  in $p$ on $(\frac{N+2}{N-2}, \infty)$.    The same    proof also shows that $ p \mapsto \nu(p)$ is analytic in $p$.

Define $ \mathcal{A}:=\{ \frac{N+2}{N-2} \le  p< p_\alpha(N): \nu(p) = - \lambda_i \mbox{ \; for some $i \ge 1$} \}$.  Note we are omitting the values of $ 1 < p < \frac{N+2}{N-2}$ since one can find a positive solution of (\ref{dom_pe}) using the standard variational approach.

 \begin{cor} There exists some $ \frac{N+2}{N-2} \le p_1 < p_2 <p_3<...$ with $ p_k \nearrow p_\alpha(N)$ such that $ \mathcal{A} \subset \{ p_k: k \ge 1\}$.  Hence for all $ \frac{N+2}{N-2} \le  p < p_\alpha(N)$, which avoids $\{p_k\}_k$, then there is no non zero solution of (\ref{deg}).

\end{cor}

 \begin{proof}  We begin by showing that $ \nu(p)$ is bounded on any interval of the form $ [1, p_\alpha(N)-\E)$ where $ \E>0$ is small.  To see this first recall from the proof of Proposition \ref{anal} we have $ \|v_p\|_{L^\infty}^{p-1} $ bounded on the interval $ (1,1+\E)$ for $ \E>0$
  and small.   One can also see that $ \|v_p\|_{L^\infty}^{p-1}$ is bounded provided we stay away from the new critical exponent $ p_\alpha(N)$.  Hence for $ \E>0$ small there is some $ C_\E>0$ such that $ \|v_p\|_{L^\infty}^{p-1} \le C_\E$ for all $ 1 < p < p_\alpha(N)-\E$. So from this we have
 \begin{eqnarray*}
 - \nu(p)& \le & \sup_{ \psi \in D(\tilde{L})} \left(  \frac{ - \int_{B} | \nabla \psi|^2 + p C_\E \int_{B} r^\alpha \psi^2 }{ \int_{B} \frac{\psi^2}{|x|^2}} \right) \\
 & \le &  \sup_{ \psi \in D(\tilde{L})} \left( \frac{ p C_\E \int_B r^\alpha \psi^2 }{\int_B \frac{\psi^2}{|x|^2}} \right) \\
 & \le & p C_\E,
 \end{eqnarray*} and recalling that $ \nu(p)$ is negative gives the desired result.    We now suppose  that $ \mathcal{A}$ has an accumulation point in $ [\frac{N+2}{N-2},p_\alpha(N))$,   say at some $ \frac{N+2}{N-2} \le  p^* < p_\alpha(N)$.  This implies there is some infinite sequence of distinct points $ q_k \in \mathcal{A}$ with $ q_k \rightarrow p^*$.   So for each $k$ there is some positive integer $\sigma(k) $ such that $ \nu(q_k)=-\lambda_{\sigma(k)}$.  But since $ \nu$ is bounded on  on $[1, p_\alpha(N)-\E)$ we see there must be some integer $i$ such that $ \sigma(k)=i$ for an infinite number of $k$ and hence $ \nu(q_k)= - \lambda_i$ for an infinite number of $k$.  But since $ p \mapsto \nu(p)+ \lambda_i$ is a real analytic function defined on $ (1,p_\alpha(N))$ and  $ p^* \in (1,p_\alpha(N))$, we can conclude that $ \nu(p)+ \lambda_i$ is identically zero on $(1,p_\alpha(N))$; which gives us the desired contradiction.

 \end{proof}

 \textbf{Proof of Theorem \ref{main_2}.}  To find a positive classical \emph{fast decay} solution of $ -\Delta u = u^p$ in $ \Omega$ with $ u=0$ on $ \pOm$ where $ \Omega$ is an exterior domain not containing the origin it is sufficient to find a positive classical solution of
 \begin{equation} \label{starte}
  -\Delta v =|x|^{p(N-2)-N-2} v^p \quad   \mbox{  in } \quad \tilde{\Omega}  \qquad \mbox{ with } \quad  v=0 \mbox{  on }  \quad  \partial \tilde{\Omega}
  \end{equation}
    where $ \tilde{\Omega}$ is  the Kelvin transform on $ \Omega$.    Let $ v_p$ denote the radial solution on the ball and let $ \nu_1(p)$ denote the quantity defined above.  From \cite{delpino_small} we know that $ \nu_1(p)$ is non constant analytic function defined on $ (\frac{N+2}{N-2},\infty)$.  Then, as we argued before, the radial solution is non degenerate except for possibly a countable number of $p$.    Provided the radial solution is non degenerate we can find a positive classical solution of (\ref{main_2}) provided $ \tilde{\Omega}$ is a sufficiently small perturbation of the unit ball.

 \hfill $\Box$

  \section{Appendix}

   Given a radial function we define the $m$ dimensional Laplacian by
\[ \Delta_m v(r)= v''(r) + \frac{m-1}{r} v'(r).\]  Note this is well defined for fractional dimensions.
 The following theorem gives the precise change of variables result,  which has been modified for our particular nonlinearity.  We remark this change of variables was independently noticed in \cite{Grossi_2}.
 \begin{thm*} \textbf{B.} \label{late}   \cite{MEMS_pull}.
 For any $\alpha >-2$, the change of variable $u(r)= (1+\frac{\alpha}{2})^\frac{2}{p-1}\tilde{u}(r^{1+\frac{\alpha}{2}})$ gives a  correspondence between the radially symmetric solutions of the equation
 \begin{equation}
\left\{ \begin{array}{ll} \label{N}
-\Delta_N u =   |x|^\alpha u^p  & \hbox{in } B,\\
 \quad \quad \,\, u=0 & \hbox{on }\partial B, \end{array} \right.
\end{equation}
in dimension $N$ and those of the equation
 \begin{equation}
\left\{ \begin{array}{ll}\label{Nalpha}
-\Delta_{N(\alpha)} \tilde{u} =   \tilde{u}^p & \hbox{in } \tilde{B},\\
\,\, \qquad  \qquad \tilde{u}=0 & \hbox{on }\partial \tilde{B}, \end{array} \right.
\end{equation}
in -- the potentially fractional -- dimension $N(\alpha)=\frac{2(N+\alpha)}{2+\alpha}$.
\end{thm*}

\begin{proof} A computation shows that
\[ \Delta_N u(r) + r^\alpha u(r)^p =(1+\frac{\alpha}{2})^\frac{2p}{p-1} r^\alpha \left( \Delta_{N(\alpha)} \tilde{u}(s)\big|_{s=r^{\frac{\alpha}{2}+1}} + \tilde{u}(r^{\frac{\alpha}{2}+1})^p \right),\] and the desired result easily follows.

\end{proof}

 Using the same change of variables gives a correspondence between solutions of the linearized equations.
\begin{cor} \label{linearized_change_cor}  Let $ u(r)$ denote a classical positive radial solution of (\ref{N}) and let $ \tilde{u}$ be as in the above theorem.  Let $ a(r)$ denote a solution of $ -\Delta a(r) = p r^\alpha u(r)^{p-1} a(r)$ on the unit ball in $ \IR^N$ with $ a(1)=0$.  Define $ \tilde{a}$ by $ a(r)=\tilde{a}(r^{\frac{\alpha}{2}+1})$.  Then
\begin{equation} \label{linearized_change}
-\Delta_{N(\alpha)} \tilde{a}(s) = p \tilde{u}(s)^{p-1} \tilde{a}(s)
\end{equation} for all $ 0 < s <1$ with $ \tilde{a}(1)=0$.

\end{cor}

We now indicate how Theorem B can be used to prove Theorem A. Let $ 1<p<p_\alpha(N)$. To find a positive classical radial solution of (\ref{dom_pe})
  it is sufficient to find a positive solution $\tilde{u}$ of (\ref{Nalpha}).   So we need to find a positive radial function $ w$ which satisfies
\[ -\Delta_{N(\alpha)} w = w^p \qquad \tilde{B}, \] with zero Dirichlet boundary conditions.   If we omit the possible difficulties associated with fractional dimensions, then this equation is subcritical exactly when
\[ 1 <p < \frac{N(\alpha)+2}{N(\alpha)-2} = \frac{N+ 2 + 2\alpha}{N-2}= p_\alpha(N), \]  and we expect the standard variational approach works.   It turns out that the fractional dimensions do not pose any issues, see \cite{MEMS_pull} for more details.

 \end{document}